\documentclass[10pt, reqno]{amsart}

\usepackage{amsmath}
\usepackage{amsfonts}
\usepackage{amssymb}
\usepackage{xcolor}
\usepackage{soul}
\usepackage{amssymb,latexsym,amsmath,amsfonts}
\usepackage{hyperref}
\usepackage{amsthm}
\usepackage{dsfont}
\usepackage[mathscr]{eucal}
\usepackage[english]{babel}
\usepackage{lipsum}
\usepackage{setspace}
\usepackage{mathtools}
\usepackage{tikz}
\usepackage{pgfplots}
\usetikzlibrary{calc,angles,positioning,intersections,quotes,decorations}

\pgfplotsset{compat=1.18}
\theoremstyle{definition}
\newtheorem*{defn}{Definition}
\newtheorem{theorem}{Theorem}[section]
\newtheorem{lemma}{ \bf Lemma}[section]
\newtheorem{proposition}{\bf Proposition}[section]
\newtheorem{corollary}{\bf Corollary}[section]

\theoremstyle{remark}
\newtheorem{remark}{Remark}[section]
\newcommand{\be}{\begin{equation}}
\newcommand{\ee}{\end{equation}}
\newcommand{\Bea}{\begin{eqnarray*}}
\newcommand{\Eea}{\end{eqnarray*}}
\newcommand{\bea}{\begin{eqnarray}}
\newcommand{\eea}{\end{eqnarray}}

\numberwithin{equation}{section}

\title{Existence of exhaustions with controlled sign of mean curvature in noncompact manifolds}
\author{Anushree Das}
\date{October 2026}

\address{Beijing International Center for Mathematical Research, Peking University, Beijing, 100871, P. R. China}
\email{anushree@bicmr.pku.edu.cn}


\begin{document}

\begin{abstract}
Given a complete and non-compact Riemannian manifold of dimension $3\le n\le 7$, we find conditions for the existence of compact exhaustions $\{U_i\}_{i=1}^\infty$ such that all boundaries $\partial U_i$ have the same sign of mean curvature. In particular, we show that in these dimensions, a manifold with one end satisfying a no tentacle condition always has a weakly mean convex, weakly mean concave, or minimal exhaustion, according to the behaviour of the area profile of separating hypersurfaces. This is a version of a result sketched by Gromov, under a condition which is weaker than his thickness at infinity. In more general situations, we study the existence of exhaustions which are weakly mean concave or weakly mean convex apart from an arbitrarily small $L^1$ error. That is, we find sufficient conditions for the existence of exhaustions $\{U_i\}_{i=1}^\infty$ which satisfy $\int_{\partial U_i}H^+\to0$, $\int_{\partial U_i}H^-\to0$, or $\int_{\partial U_i}|H|\to0$, where $H^-$, $H^+$, and $|H|$ are the negative part, positive part, and the absolute value of the mean curvature respectively. As an application of our methods, we study the existence of almost $c$-convex exhaustions, that is, exhaustions whose boundaries satisfy $\int_{\partial U_i}(H-c)^-\to0$ for some $c>0$.
\end{abstract}

\maketitle

\section{Introduction}

For an open Riemannian manifold $M$ with a metric $g$, we say that a nested sequence of submanifolds $\{U_i\}_{i=1}^\infty$ is a compact exhaustion of $M$ if $U_i \subset \mathring U_{i+1}$, $M = \bigcup^\infty_{i=1}U_i$, and each $U_i$ is a compact codimension-$0$ smooth submanifold with smooth boundary $\partial U_i$. Here, $\mathring U$ denotes the interior of a set $U$. If each $\partial U_i$ is a minimal hypersurface, we say that $\{U_i\}$ is an exhaustion with minimal hypersurface boundaries. Similarly, if each $\partial U_i$ is (weakly) mean convex or mean concave, we say that $\{U_i\}$ is an exhaustion with (weakly) mean convex or mean concave boundaries, respectively.

Throughout this article, the mean curvature $H_X$ of a hypersurface $X$ bounding a domain is computed with respect to the outward normal. When there is no ambiguity about the hypersurface, we denote the mean curvature simply by $H$. We write $H^\pm=\max\{\pm H,0\}$.

The question of when an open Riemannian manifold admits an exhaustion with mean convex boundaries has been a longstanding open problem in geometry, for instance mentioned by Lawson in \cite{book}. Hadamard manifolds and manifolds with non-negative sectional curvature both admit exhaustions with mean convex boundaries, as do asymptotically flat manifolds. Gromov's trichotomy result in \cite{Gromov} (Section $1.4$) shows that the absence of any mean convex exhaustion guarantees the presence of certain minimal hypersurfaces in the manifold. One way to get weakly mean convex exhaustions is via the long time existence of the inverse mean curvature flow, for instance, following Xu in \cite{Xu}. The consequences of the existence of a mean convex exhaustion have also been widely studied, for instance in \cite{FM}. On the other hand, it is well known that when all the ends of an open manifold are cusps, the manifold admits an exhaustion with mean concave boundaries. Recently, in \cite{BZ}, Bi and Zhu showed that a non-compact manifold with sublinear volume growth always admits a mean concave exhaustion. Their methods rely on $\mu$-bubbles and on the existence of a minimiser for a variational problem. Our ideas shall be closely related to a similar method.

A manifold of dimension $n$ is said to be \textit{thick at infinity} if it contains no non-compact minimal hypersurface with compact boundary and with finite $n-1$ volume. In \cite{Gromov}, Gromov sketched a proof of the following result.

\begin{theorem}
    Let $X$ be a Riemannian manifold possibly with compact boundary $\partial X$ and a single end such that $X$ is thick at infinity. Then one of the following possibilities is realised.
    \begin{enumerate}
        \item $X$ can be exhausted by compact strictly mean convex bubbles,
        \item $X$ can be exhausted by compact strictly mean concave bubbles,
        \item there exists a continuous positive proper function $h:X\to \mathbb{R}^+$ such that the levels $Y_t=h^{-1}(t)\subset X$ are minimal hypersurfaces, that are the boundaries of $0$-bubbles, for all $t\geq t_0\geq 0$.
    \end{enumerate}
\end{theorem}

We prove a version of Gromov's result under slightly weaker assumptions, in terms of the area profile of separating hypersurfaces near the end. In particular, we give conditions for the existence of mean convex, mean concave, and minimal exhaustions in terms of the area profile. 

Fix a point $p\in M$, and write $B_r=B(p,r)$ and $d=d(p,\cdot)$. For a compact set $K\subset M$, let
$$
A(K)=\inf\bigl\{|\partial\Omega|:\ \Omega\ \text{is a compact domain with smooth boundary and}\ K\subset\mathring\Omega\bigr\}.
$$
Also, we denote $A_\infty=\lim_{r\to\infty}A(B_r)$. Then $A(K)$ is non-decreasing in $K$, and $A(K)\le A_\infty$.

\begin{defn}
Let $(M,g)$ be a complete non-compact Riemannian manifold, and let $K\subset M$ be compact. A \textit{tentacle for} $K$ is defined as a set $\Omega\supset K$ of locally finite perimeter which is the closure of an open set, satisfying the following properties:
\begin{enumerate}
    \item $\partial \Omega$ is unbounded;
    \item $|\partial \Omega|\leq A_\infty$ and $|\partial \Omega|<\infty$;
    \item $ \Omega$ is perimeter minimising amongst all compact perturbations containing $K$. That is, for any $\Omega'\supset K$ with $\Omega'\Delta \Omega$ bounded, $|\partial \Omega|\leq |\partial \Omega'|$.
\end{enumerate}

Here $|\partial \Omega|$ denotes the perimeter of $\Omega$. The manifold $(M,g)$ is said to satisfy the \textit{no tentacle} condition (henceforth called the $NT$ condition) if there exists a compact set $K_0$ such that for any compact set $K\supset K_0$, $K$ has no tentacle.
\end{defn}

The $NT$ condition rules out manifolds where the minimum area hypersurfaces bounding any domain containing some compact set are given by an unbounded embedded hypersurface of small area going into the end, and is a natural condition arising in many geometric scenarios. Thickness at infinity implies $NT$ condition, however the converse is not true. This is because the $NT$ condition only rules out those hypersurfaces which are obtained as minimisers. For details, we refer to the example at the end of Section \ref{sectiontrichotomy}. In particular, if a manifold has bounded geometry (the injectivity radius is bounded below by a constant $\iota>0$, and the absolute value of the sectional curvature is bounded above by a constant $\kappa_0$), then the manifold satisfies the $NT$ condition.

In the rest of the paper, we shall assume that when the manifold satisfies the $NT$ condition, the compact sets we consider are large enough to have no tentacle (that is, we always consider compact sets $K$ such that $K\supset K_0$). If a manifold with a single end satisfies the $NT$ condition, our first result gives conditions for the existence of a (weak) mean convex, a mean concave, or a minimal boundary exhaustion.

\begin{theorem}\label{theoremtrichotomy}
Let $(M,g)$ be a complete Riemannian manifold of dimension $3\le n\le 7$ which has a single end and satisfies condition $NT$. Then exactly one of the following three cases occurs. 
\begin{enumerate}
\item[(1)] \emph{Strict profile:} $A(K)<A_\infty$ for every compact set $K\subset M$. Then $M$ has a compact exhaustion $\{U_i\}_{i=1}^\infty$ with smooth boundaries such that each component of $\partial U_i$ is either strictly mean convex or minimal.
\item[(2)] \emph{Flat profile:} there is a compact set $K_0$ with $A(K_0)=A_\infty$, and $A(K)$ is not attained for any compact set $K\supseteq K_0$. Then $M$ has a compact exhaustion $\{U_i\}_{i=1}^\infty$ with smooth boundaries such that each component of $\partial U_i$ is either strictly mean concave or minimal.
\item[(3)] \emph{Attained profile:} there is a compact set $K_0$ with $A(K_0)=A_\infty$, and $A(K)$ is attained for arbitrarily large compact sets $K$. Then $M$ has a compact exhaustion $\{U_i\}_{i=1}^\infty$ such that each $\partial U_i$ is a smooth stable minimal hypersurface.
\end{enumerate}
The conclusion in case (3) holds without condition $NT$ and without assuming a single end. If $A_\infty=0$, then case (2) occurs. In that case the boundaries $\partial U_i$ can be taken to be strictly mean concave, again without condition $NT$ or a single end.
\end{theorem}

The exhaustion can be chosen such that the boundary hypersurfaces are bubbles. In particular, if the manifold has uniformly positive scalar curvature, then the exhaustion can further be chosen such that the boundary hypersurfaces also admit some positive scalar curvature metric, for instance by a suitable $\mu$-bubble argument. A manifold may admit different exhaustions with opposite signs of mean curvatures on their boundaries. However, the presence of one weakly mean concave exhaustion and one weakly mean convex exhaustion implies the presence of a minimal exhaustion (for instance, by finding a minimal hypersurface in a band bounded by a mean concave and mean convex hypersurface). Hence, if more than one of above kinds of exhaustions exist in the same manifold, then case $(3)$ must hold true.

Note that the failure of the above mentioned trichotomy implies that all end-separating hypersurfaces far enough towards the end have mixed sign of mean curvature. If the manifold has multiple ends, such that each end satisfies an equivalent $NT$ condition, then the trichotomy holds separately for each end. In particular, we cannot guarantee that the sign of the exhaustion in each end remains the same. The standard example for this situation is a manifold with one cusp like end and one funnel like end, where the boundary of an exhaustion is mean concave near the cusp, and mean convex near the funnel.

Since manifolds with bounded geometry satisfy the $NT$ condition, we have the following corollary.

\begin{corollary}\label{cor:boundedgeometry}
    Let $(M,g)$ be a complete one-ended Riemannian manifold of dimension $3\leq n \leq 7$, such that it has bounded geometry. Then, $M$ has a compact exhaustion such that all boundaries of the exhaustion are weakly mean convex hypersurfaces, all boundaries are weakly mean concave hypersurfaces, or all boundaries are minimal hypersurfaces.
\end{corollary}

Cheeger and Gromov showed, in \cite{CG}, that a non-compact manifold of finite volume with the absolute value of sectional curvature bounded above admits an exhaustion $\{U_i\}_{i=1}^\infty$ such that the volume of $\partial U_i\to 0$ as $i\to \infty$, and the second fundamental form of $\partial U_i$ are uniformly bounded. In particular, the total mean curvature of the boundaries, which is defined as $\int_{\partial U_i}|H_{\partial U_i}|$, tends to $0$ as $i\to \infty$.

In the absence of the $NT$ condition, we investigate the existence of exhaustions where the total mean curvature tends to $0$. More generally, we also find some conditions for the existence of an exhaustion where the integral of the negative mean curvature (resp. positive mean curvature) of the boundary tends to $0$, that is, an exhaustion which has almost weakly mean convex (resp. weakly mean concave) boundary with arbitrarily small $L^1$ error. We will need the following constant, which measures how fast the small geodesic balls around a point collapse as we go out to the end of the manifold.

\begin{defn}
For a Riemannian manifold $(M,g)$ and a point $q$ in $M$, let $\tilde r(q)$ be the largest radius $r\le1$ such that $B(q,\tilde r(q))$ admits a coordinate chart in which $\frac14\delta_{ij}\le g_{ij}\le4\delta_{ij}$. We define the \textit{exponential rate of collapse} of $M$ as
\begin{align}
\kappa=\limsup_{d(q)\to\infty}\frac{\log(1/\tilde r(q))}{d(q)}\in[0,\infty]
\end{align}
If $\tilde r(q)\ge r_0>0$ for all $q\in M$, then $\kappa=0$, and we say that $M$ is non-collapsed.
\end{defn}

This holds, for instance, under a two-sided bound on the sectional curvature and a lower bound on the injectivity radius. Manifolds which are non-collapsed satisfy the $NT$ condition.
The main result of this approach is the following.

\begin{theorem}\label{cor:L1}
Let $(M,g)$ be a complete Riemannian manifold of dimension $3\le n\le7$.
\begin{enumerate}
\item If $A_\infty<\infty$, there is a compact exhaustion $\{U_i\}$ of $M$ with smooth boundaries such that $\int_{\partial U_i}H^+\to0$.
\item If $M$ has a single end, $A_\infty>0$ and $\kappa<\infty$, there is a compact exhaustion $\{U_i\}$ of $M$ with smooth boundaries such that $\int_{\partial U_i}H^-\to0$.
\item If $M$ has a single end, $0<A_\infty<\infty$ and $\kappa<\infty$, there is a compact exhaustion $\{U_i\}$ of $M$ with smooth boundaries such that $\int_{\partial U_i}|H|\to0$.
\item If $M$ is non-collapsed, there is a compact exhaustion $\{U_i\}$ of $M$ with smooth boundaries such that $\int_{\partial U_i}H^-\to0$, and such that $\int_{\partial U_i}|H|\to0$ if moreover $A_\infty<\infty$.
\end{enumerate}
\end{theorem}

Manifolds with bounded geometry are non-collapsed, and hence satisfy part $(4)$ of the previous theorem. Thus if a manifold with bounded geometry has $A_\infty<\infty$, there exists an exhaustion with arbitrarily small total mean curvature.

Finally, we apply our techniques to study the existence of exhaustions whose boundaries have mean curvature lower bounds close to a positive constant. We define the following linear isoperimetric condition on manifolds.

\begin{defn}
We say that $(M,g)$ satisfies a \emph{linear isoperimetric inequality with constant $c'>0$} if there is $C'\ge0$ such that $|\partial\Omega|\ge c'|\Omega|-C'$ for every compact domain $\Omega$.
\end{defn}

When the manifold satisfies a linear isoperimetric inequality with a constant $c'>0$, then under mild additional conditions, we get an exhaustion with boundaries having almost a $c$ lower bound of mean curvature for any $c<c'$. We refer to Theorem \ref{thm:cconvex} for the precise result. In particular, we have the following corollary.

\begin{corollary}\label{cor:cconvex}
Let $(M,g)$ be a complete Riemannian manifold of dimension $3\le n\le7$ satisfying a linear isoperimetric inequality with constant $c'>0$. Suppose that $M$ is non-collapsed, or that $M$ has a single end, infinite volume, and $\kappa<\infty$. Then for every $c\in(0,c')$ there exists a compact exhaustion $\{U_i\}$ of $M$ with smooth boundaries such that $\int_{\partial U_i}(H-c)^-\to0$.
\end{corollary}

Section \ref{sectiontrichotomy} of the paper is devoted to the trichotomy results, and Section \ref{section:smallcurv} discusses the existence of exhaustions with boundaries having small errors of curvature. Subsection \ref{subsec:smallcurv} concerns exhaustions with boundaries which are almost (weakly) mean convex or mean concave with small $L^1$ errors. Subsection \ref{sectioncurvatureconditions} studies the necessary curvature conditions on the manifolds and derives pointwise bounds on mean curvature of exhaustion boundaries. Finally, Subsection \ref{subsec:cconvex} discusses the existence of exhaustions with boundaries having positive lower bound of curvature with small errors.

\subsection*{Acknowledgement} The author would like to thank Yuguang Shi for his support.

\subsection*{AI Statement} The ideas and the proofs in this paper are the author's. Claude Opus 5 was used to verify the proofs, to search relevant literature, and to simplify and draft certain arguments. The same AI agent was used to check for errors during the writing process. The author remains responsible for the contents and writing of the paper.

\section{A trichotomy result on the mean curvature of an exhaustion}\label{sectiontrichotomy}

In this section we establish the trichotomy theorem regarding the existence of exhaustions in a non-compact manifold $M$ with a complete Riemannian metric. We begin by showing that if the manifold has $A_\infty =0$ or $A_\infty=\infty$, then it admits exhaustions with (weak) mean concave or mean convex boundaries, respectively.

\begin{proposition}\label{propAinfty0}
    Let $(M,g)$ be a complete Riemannian manifold of dimension $3\leq n \leq 7$, such that $A_\infty = 0$. Then there exists a compact exhaustion $\{U_i\}_{i=1}^\infty$ of $M$ with smooth boundaries such that each $\partial U_i$ is strictly mean concave.
\end{proposition}

\begin{proof}
    The proof is identical to that of Theorem $1.2$ in \cite{BZ}. Since $A_\infty = 0$, we get hypersurfaces $X_i$ that separate the ends from a compact domain, and $\lvert X_i\rvert \to 0$ as $i\to\infty$. Taking a compact set $K$ with boundary $\partial K$, we can choose an outer wall $X_i$ with sufficiently small area, so that the $\mu$-bubble construction in \cite{BZ} holds. This gives us a smooth hypersurface $\xi_1$ that separates the ends from $K$ and has strictly negative mean curvature. Choose $U_1$ to be the compact domain bounded by $\xi_1$. Now, choose a larger domain $K'$ strictly containing $U_1$ in its interior and repeat the process to get a mean concave smooth hypersurface $\xi_2$, and the compact domain $U_2$. On iterating, we get a compact exhaustion of $M$ with strictly mean concave boundaries.
\end{proof}

\begin{proposition}\label{propAinftyinfty}
    Let $(M,g)$ be a complete, single-ended Riemannian manifold of dimension $3\leq n \leq 7$, such that $A_\infty = \infty$ and $M$ satisfies condition $NT$. Then there exists a compact exhaustion $\{U_i\}_{i=1}^\infty$ of $M$ with smooth boundaries such that each component of $\partial U_i$ is either strictly mean convex or minimal. In particular, we get an exhaustion $U_i$ such that either each $\partial U_i$ has a strictly mean convex component, or each $\partial U_i$ is minimal.
\end{proposition}

\begin{proof}
     Take a compact domain $K\subset M$ with smooth boundary such that $M\setminus{K}$ is connected and unbounded. This is possible because $M$ has one end. Consider a larger compact domain $V$ such that $K$ lies in the interior of $V$, and $X=V\setminus \mathring{K}$ has a single connected component. Then $\partial V$ and $\partial K$ together bound the band $X$. Let $\mathcal{C}(X,\partial V,\partial K)$ denote the class of those smooth closed hypersurfaces in $X$ which separate the boundaries $\partial V$ and $\partial K$. Minimise the area amongst all hypersurfaces in $\mathcal{C}(X,\partial V,\partial K)$. Since the dimension of the ambient manifold is $3\leq n \leq 7$, the minimiser exists and is a $C^{1,1}$ hypersurface $\xi$ which is smooth away from the boundaries of $X$. 

     We first assume that $\xi$ lies entirely in the interior of $X$. Then, $\xi$ is a local area minimiser in $M$ and hence is a stable minimal hypersurface. If instead $\xi$ touches $\partial K$ but not $\partial V$, then the mean curvature is non-negative a.e. on $\xi\cap \partial K$. In $\xi\cap \mathring{X}$, $\xi $ is smooth and has $0$ mean curvature. Therefore, $\xi$ is a $C^{1,1}$ hypersurface with nonnegative mean curvature. Running the mean curvature flow on $\xi$ for a very short time gives us a smooth weakly mean convex hypersurface $\xi'$, such that on those components where $\xi'$ is not minimal, it is strictly mean convex.

     Now, the other possibility is that $\xi$ touches $\partial V$ and is not a local area minimiser in $M$. Then, the mean curvature is non-positive a.e. on $\xi\cap \partial V$. Consider a compact domain $V'\supset V$ such that $\xi$ is in the interior of $V'$, and the band $X'=V'\setminus \mathring{K}$. We again minimise the area in the class $\mathcal{C}(X',\partial V',\partial K)$, and get a $C^{1,1}$ minimising hypersurface $\xi'$. If $\xi'$ does not touch $\partial V'$, then we get a minimal or a weakly mean convex hypersurface as before. Hence, we can again assume that $\xi'$ touches $\partial V'$.

     Now, we keep taking an increasing sequence of compact domains $\{V_i\}$ exhausting $M$ such that ${V}'\subset \mathring V_1$, and ${V_i}\subset \mathring V_{i+1}$. We similarly find the area minimiser in each band $X_i=V_i\setminus \mathring{K}$, and get a sequence of $C^{1,1}$ hypersurfaces $\{\xi_i\}$. We assume that any minimiser in $X_i$ touches $\partial V_i$, since otherwise we obtain a weakly mean convex or minimal hypersurface.

     Consider the sequence $\{\xi_i\}$. Since $X_i\subset X_{i+1}$ for all $i$, each $\xi_i$ is a competitor for any $\xi_j$, $j>i$. Hence, $\lvert \xi_j\rvert\leq\vert \xi_i \rvert$ for $j>i$, and the sequence has decreasing area. Assume that there exists a compact set $Z$ and a subsequence $\{\xi_{i_k}\}$ such that each member of the subsequence intersects $Z$. Then we consider the limit of this subsequence, which must also intersect $Z$ by a standard density argument. The limit must also be locally minimising. First, assume that the limit is either an unbounded hypersurface in $M$, or a bounded hypersurface of area not greater than any $|\xi_i|$. In the first case, since the sequence $\{\xi_{i}\}$ has decreasing area, the limit hypersurface has area less than $A_\infty=\infty$. This contradicts condition $NT$. If instead the limit is bounded, then consider the case that it separates $K$ from the end. Then, it lies in the interior of some $V_i$, and we get a contradiction to every minimiser in $\mathcal{C}(X_i,\partial V_i,\partial K)$ touching $\partial V_i$. Otherwise, since the manifold has a single end, the limit is a bounded hypersurface which does not separate $K$ from the end, and hence it bounds some compact set $W$ which is disjoint from $\mathring K$. In that case, filling in $W$ gives a competitor of lesser area, contradicting the minimality of the limit.

     Therefore, any infinite subsequence of $\{\xi_i\}$ must eventually entirely leave any compact domain in $M$. Then we get an infinite sequence of end-separating hypersurfaces in $M$ with area bounded above, contradicting the fact that $A_\infty =\infty$.
     
     Therefore, we must eventually find a $V_i$ and a minimiser $\xi_i$ such that $\xi_i$ does not intersect $\partial V_i$. Then, $\xi_i$ is either minimal, or weakly mean convex with at least one strictly mean convex component. Now, we choose a new compact domain $K'$ containing $V$ such that the compact domain bounded by $\xi_i$ lies in the interior of $K'$, and repeat the process. We continue iteratively by choosing a sequence of domains $\{K'_i\}$ such that they exhaust $M$. 

     Thus, we get an infinite nested sequence $\{\zeta_i\}$ of smooth closed hypersurfaces that separate the end of $M$ from a compact domain. Each $\zeta_i$ is either minimal, or weakly mean convex with at least one strictly mean convex component. Additionally, each $\zeta_i$ bounds a compact domain $U_i$ of $M$. If there are infinitely many minimal hypersurfaces in $\{\zeta_i\}$, we discard the hypersurfaces with a mean convex component to get an infinite sequence $\{\zeta_{i_j}\}$ of minimal hypersurfaces. Then, taking the corresponding domains $\{U_{i_j}\}$, we get an exhaustion of $M$ with minimal boundary. If not, then discarding the finitely many minimal elements of $\{\zeta_i\}$, we get an infinite subsequence of weakly mean convex boundaries with at least one strictly mean convex component. Taking the corresponding compact domains bounded by the elements of this subsequence then gives us an exhaustion of $M$ with weakly mean convex boundaries and at least one strictly mean convex component.
      \end{proof}

\begin{remark}
    While sublinear growth guarantees $A_\infty=0$ and hence there exists a mean concave exhaustion, superlinear growth is not enough to guarantee a mean convex exhaustion.
\end{remark}

We now show that the assumptions in Propositions \ref{propAinfty0} and \ref{propAinftyinfty} can be further weakened to the situations where the minimum separating hypersurfaces have increasing or decreasing area.

\begin{proposition}\label{prop:strict}
Let $(M,g)$ be a complete, single-ended Riemannian manifold of dimension $3\le n\le 7$ which satisfies condition $NT$. Suppose that $A(K)<A_\infty$ for every compact set $K\subset M$; equivalently, every compact set $K$ is contained in a compact set $K'$ with $A(K')>A(K)$. Then there exists a compact exhaustion $\{U_i\}_{i=1}^\infty$ of $M$ with smooth boundaries such that each component of $\partial U_i$ is either strictly mean convex or minimal. In particular, either each $\partial U_i$ has a strictly mean convex component, or each $\partial U_i$ is minimal.
\end{proposition}

\begin{proof}
We follow the proof of Proposition \ref{propAinftyinfty}. In that proof, the assumption $A_\infty=\infty$ was used to rule out an unbounded limit hypersurface by condition $NT$, and to rule out the case where the minimisers $\xi_i$ eventually leave every compact set. We show that the assumption that every compact set $K$ is contained in a compact set $K'$ with $A(K')>A(K)$ is enough to rule out these two situations. The rest of the proof is unchanged.

We consider the minimisers $\xi_i$ in $\mathcal{C}(X_i,\partial V_i,\partial K)$ as before. First, the sequence $\{\xi_i\}$ satisfies $|\xi_i|\to A(K)$. Indeed, $|\xi_i|\ge A(K)$ for every $i$. Conversely, any compact domain $\Omega$ with smooth boundary and $K\subset\mathring{\Omega}$ lies in $V_i$ for large $i$, since the $V_i$ exhaust $M$, and then $|\xi_i|\le|\partial\Omega|$.

Suppose a subsequence meeting a fixed compact set $Z$ converges to an unbounded hypersurface. By the previous paragraph, its area is at most $\lim_i|\xi_i|=A(K)\le A_\infty$, and in particular it is finite. It therefore contradicts condition $NT$, exactly as before. The two cases of a bounded limit, according to whether or not it separates $K$ from the end, are ruled out by the same argument as in Proposition \ref{propAinftyinfty}.

Now, suppose the $\xi_i$ eventually leave every compact set. Fix a point $p$ in $K$, and a radius $r$ with $K\subset B_r=B(p,r)$. For $i$ large, $\xi_i$ does not meet $\overline{B}_r$, the closure of $B_r$. Since $\overline{B}_r$ is connected and contains $K$, it then lies in the interior of the domain bounded by $\xi_i$, so $|\xi_i|\ge A(B_r)$. Letting $i\to\infty$ gives $A(K)\ge A(B_r)$ for every $r$, hence $A(K)\ge A_\infty$. This contradicts the hypothesis, and replaces the contradiction with $A_\infty=\infty$ in Proposition \ref{propAinftyinfty}.

Hence, as before, for some $i$ there is a minimiser $\xi_i$ that does not touch $\partial V_i$. It is either minimal, or weakly mean convex with a strictly mean convex component. The hypothesis holds for every compact set, so we can iterate with an exhausting sequence of domains as before to obtain the required exhaustion.
\end{proof}

\begin{remark}
The hypothesis $A(K)<A_\infty$ cannot be dropped. Consider an end whose cross-sections decrease to a positive limit. Then $A(K)=A_\infty$ for all large $K$. By comparison with the cross-sections at an innermost point, every far separating hypersurface has a point of negative mean curvature, so no weakly mean convex exhaustion exists.
\end{remark}

\begin{proposition}\label{prop:flat}
Let $(M,g)$ be a complete, single-ended Riemannian manifold of dimension $3\le n\le 7$ which satisfies condition $NT$. Suppose that there is a compact set $K_0$ with $A(K_0)=A_\infty$. Suppose also that for every compact set $K\supseteq K_0$ the infimum $A(K)$ is not attained, that is, there is no bounded set $\Omega\supseteq K$ of finite perimeter with $|\partial\Omega|=A(K)$. Then there exists a compact exhaustion $\{U_i\}_{i=1}^\infty$ of $M$ with smooth boundaries such that each component of $\partial U_i$ is either strictly mean concave or minimal. In particular, either each $\partial U_i$ has a strictly mean concave component, or each $\partial U_i$ is minimal.
\end{proposition}

Since $A(K)$ is non-decreasing in $K$ and bounded by $A_\infty$, the first hypothesis gives $A(K)=A_\infty$ for every compact $K\supseteq K_0$. The hypotheses hold automatically when $A_\infty=0$. In that case $NT$ is vacuous, the infimum $0$ is never attained, and Proposition \ref{propAinfty0} gives the stronger conclusion with strictly mean concave boundaries. The new case is $0<A_\infty<\infty$, with the value $A_\infty$ never attained, as for an end whose cross-sections decrease to a positive limit.

\begin{proof}
We use the band minimisers from the proof of Proposition \ref{propAinftyinfty}. Fix a compact domain $K\supset K_0$ with smooth boundary, take compact domains $V_i$ exhausting $M$ with $K\subset\mathring{V}_i$, and let $\xi_i$ be an area minimiser in $\mathcal{C}(V_i\setminus\mathring{K},\partial V_i,\partial K)$. Any compact domain with smooth boundary containing $K$ in its interior lies in $V_i$ for $i$ large, so $|\xi_i|\to A(K)=A_\infty$.

We claim that $\xi_i$ does not touch $\partial K$ for $i$ large. Suppose instead that a subsequence touches $\partial K$. Then the subsequence converges to a limit hypersurface $\xi_\infty$ that still touches $\partial K$. The limit bounds a set $\Omega_\infty\supseteq K$ with $|\xi_\infty|\le A(K)=A_\infty$, and $\Omega_\infty$ minimises area among compact perturbations containing $K$. We rule out the three possibilities for this limit, as in Proposition \ref{propAinftyinfty}.
\begin{itemize}
\item If $\xi_\infty$ is unbounded, then $\Omega_\infty$ is a tentacle for $K$, since $|\xi_\infty|\le A_\infty<\infty$. This contradicts condition $NT$.
\item If $\xi_\infty$ is bounded and separates $K$ from the end, then $\Omega_\infty$ is a bounded set containing $K$ with $|\partial\Omega_\infty|\le A(K)$. So the infimum $A(K)$ is attained, contrary to the hypothesis.
\item If $\xi_\infty$ is bounded and does not separate $K$ from the end, then, since $M$ has one end, it bounds a compact set $W$ disjoint from $\mathring{K}$. Filling in $W$ contradicts the minimality of the limit, exactly as before.
\end{itemize}

Hence, for $i$ large, $\xi_i$ touches at most the outer wall $\partial V_i$. Its mean curvature vanishes on the free part and is non-positive almost everywhere on $\xi_i\cap\partial V_i$. After a short mean curvature flow, we obtain a smooth weakly mean concave hypersurface each of whose components is either strictly mean concave or minimal. The hypothesis holds for every compact set containing $K_0$. So we can iterate, choosing the successive compact sets to exhaust $M$, and obtain the desired exhaustion.
\end{proof}

\begin{remark}
As before, the flat hypothesis cannot be dropped. For an end whose cross-sections increase to a limit, $A(K)<A_\infty$ for every $K$. By comparison with the cross-sections at an outermost point, every far separating hypersurface then has a point of positive mean curvature.
\end{remark}

The following proposition deals with the remaining case, which is standard in the literature, but we add it here for completeness.

\begin{proposition}\label{prop:attained}
Let $(M,g)$ be a complete Riemannian manifold of dimension $3\le n\le 7$. Suppose that there is a compact set $K_0$ with $A(K_0)=A_\infty$. Suppose also that every compact set is contained in a compact set $K$ for which the infimum $A(K)$ is attained, that is, there is a bounded set $\Omega\supseteq K$ of finite perimeter with $|\partial\Omega|=A(K)$. Then there exists a compact exhaustion $\{U_i\}_{i=1}^\infty$ of $M$ such that each $\partial U_i$ is a smooth stable minimal hypersurface. Moreover, each $U_i$ minimises perimeter among bounded sets containing $K_0$.
\end{proposition}

Unlike Propositions \ref{prop:strict} and \ref{prop:flat}, this needs neither condition $NT$ nor a single end.

\begin{proof}
Let $K$ be a compact set containing a neighbourhood of $K_0$ for which $A(K)$ is attained by some $\Omega$. Since $A$ is non-decreasing and bounded by $A_\infty$, we have $A(K)=A_\infty=A(K_0)$. Since $\Omega\supseteq K_0$, it follows that $\Omega$ also minimises perimeter among bounded sets containing $K_0$. Moreover, $K$ contains a neighbourhood of $K_0$, so $\partial\Omega$ stays at a positive distance $\rho$ from $K_0$. Hence perturbations of $\Omega$ supported in balls of radius less than $\rho$ centred on $\partial\Omega$ still contain $K_0$, and $\Omega$ is locally perimeter minimising in $M$. As $3\le n\le 7$, $\partial\Omega$ is a smooth stable minimal hypersurface.

To build the exhaustion, choose inductively compact sets $K_i$ containing $B_i$ and neighbourhoods of $K_0$ and $U_{i-1}$, for which $A(K_i)$ is attained. This is possible by hypothesis. Let $U_i$ be the closure of a corresponding minimiser. Then $U_{i-1}\subset\mathring{U}_i$ and $B_i\subset U_i$, so $\{U_i\}$ is a compact exhaustion of $M$. By the previous paragraph, each $\partial U_i$ is a smooth stable minimal hypersurface.
\end{proof}

We can now proceed with the proof of the existence of the trichotomy.

\subsection*{Proof of Theorem \ref{theoremtrichotomy}}
Since $A(K)$ is non-decreasing in $K$, with limit $A_\infty$ along the balls $B_r$, either $A(K)<A_\infty$ for every compact $K$, or $A(K_0)=A_\infty$ for some compact $K_0$. In the second situation there are two possibilities. Either $A(K)$ is attained for arbitrarily large compact sets $K$, which is case (3). Or there is a compact set $C$ such that $A(K)$ is not attained for any compact $K\supseteq C$. Since $A(K_0\cup C)=A_\infty$ by monotonicity, case (2) then holds with $K_0\cup C$ in place of $K_0$. The three cases are mutually exclusive.

In case (1), Proposition \ref{prop:strict} gives the exhaustion; when $A_\infty=\infty$ this is Proposition \ref{propAinftyinfty}. In case (2), Proposition \ref{prop:flat} gives it, and in case (3), Proposition \ref{prop:attained}. Finally, if $A_\infty=0$, the value $0$ is never attained, so case (2) occurs, and Proposition \ref{propAinfty0} gives strictly mean concave boundaries.

\qed

The $NT$ condition is required when $A_\infty>0$ to rule out the infimum being always achieved by a hypersurface which touches both boundaries of each band. Taking the limit would then give an unbounded hypersurface going into the end, with area lesser than the boundary of any compact domain containing the initial set $K$. This is ruled out naturally in many geometric settings, like in manifolds having bounded geometry. In particular, under bounded geometry conditions, unbounded hypersurfaces which are area minimisers must have infinite area, which gives both the thickness at infinity condition and the $NT$ condition. We prove this here for completeness.

\subsection*{Proof of Corollary \ref{cor:boundedgeometry}}
We show that a manifold with bounded geometry satisfies the $NT$ condition. The result then follows from Theorem \ref{theoremtrichotomy}. 
Let $K\subset M$ be a compact domain. Let $\Omega$ satisfying $K\subset\Omega\subset M$ be a set such that $\partial \Omega$ is smooth and unbounded, $\lvert \partial \Omega\rvert$ is finite, and $\partial \Omega$ minimises area among compact perturbations that still contain $K$. Then $\Omega$ violates condition $NT$.

Consider a compact neighbourhood $U$ of $K$. Any perturbation of $\Omega$ supported in $M\setminus U$ has boundary area at least $\lvert \partial \Omega \rvert$, since such a perturbation still contains $K$. Therefore, $\partial \Omega$ is locally perimeter minimising in $M\setminus U$, and hence is a smooth minimal hypersurface in $M\setminus U$.

Since the manifold has bounded geometry, the monotonicity formula holds with some constants depending only on dimension $n$, injectivity radius lower bound $\iota$, and sectional curvature bound $\kappa_0$. That is, there exist $r_1>0$ and $c>0$ such that for any minimal hypersurface $\Sigma$ in $M$, $\lvert \Sigma\cap B(x,r_1)\rvert\geq cr_1^{n-1}$ for any $x\in \Sigma$ such that $\Sigma$ does not have any boundary in $B(x,r_1)$.

Since $\partial \Omega$ is unbounded, we can choose an infinite sequence of points $\{x_j\}_{j=1}^\infty$, $x_j \in (\partial \Omega\setminus U)$ such that $d(x_j,U)> r_1$ for all $j$, and $d(x_i,x_j)\geq 2r_1$ for all $i\neq j$. Then, $$\lvert \partial \Omega\rvert \geq \Sigma_{j=1}^\infty \lvert \partial \Omega\cap B(x_j,r_1)\rvert\geq \Sigma_{j=1}^\infty cr_1^{n-1}=\infty.$$
This contradicts the finiteness of $\partial\Omega$, and hence such an $\Omega$ cannot exist. Therefore, a manifold with a single end, bounded geometry, and dimension $3\leq n \leq 7$ always satisfies the criteria for Theorem \ref{theoremtrichotomy}, and hence admits a trichotomy of exhaustions.
\qed

Finally, we give an example to show that the $NT$ condition is weaker than the thick at infinity condition.

\subsection*{Example}
Let $n\ge3$. Fix smooth even functions $q,\eta$ on $[-2,2]$ with the following properties:
\begin{itemize}
\item $q(s)=s^2$ and $\eta(s)=1$ for $|s|\le\frac12$;
\item $q(s)=\sin^2(2-|s|)$ near $s=\pm2$, and $q>0$ on $(-2,2)\setminus\{0\}$;
\item $0\le\eta\le1$ and $\{\eta>0\}=(-1,1)$.
\end{itemize}
Fix a smooth function $\varepsilon\colon[0,\infty)\to(0,1]$ which is constant on $[0,1]$, strictly decreasing on $[1,\infty)$, and equal to $t^{-2}$ for $t\ge2$. Set $f(t,s)=(\varepsilon(t)^2\eta(s)+q(s))^{1/2}$. Write $S^{n-1}$ as $[-2,2]\times S^{n-2}$ with $\{\pm2\}\times S^{n-2}$ collapsed to points. Then the metric
\begin{align}
g=dt^2+ds^2+f(t,s)^2g_{S^{n-2}}
\end{align}
on $[0,\infty)\times S^{n-1}$ is smooth, and it is a product for $t\le1$. Close it off at $t=0$ by a compact piece with a product collar to get a complete manifold $M$ with one end. Let $K_0$ be the compact piece together with $\{t\le1\}$, and write $S_t=\{t\}\times S^{n-1}$ and $W=\{s=0\}$.

Then $M$ satisfies $NT$. Moreover, for every compact $K\supseteq K_0$ we have 
\begin{align}
A(K)=A_\infty=|S^{n-2}|\int_{-2}^2\sqrt{q}^{\,n-2}\,ds,
\end{align}
and this value is not attained. So $M$ falls under case (2) of Theorem \ref{theoremtrichotomy}, and that case is realised by the slices $S_t$, whose mean curvature $(n-2)\,\partial_tf/f$ is non-positive. However, $M$ is not thick at infinity: $W\cap\{t\ge1\}$ is a non-compact minimal hypersurface of finite area with compact boundary, and it is locally area minimising at infinity.

Indeed, $W$ is totally geodesic, being the fixed set of the isometry $s\mapsto-s$. Its area beyond $t=1$ is $|S^{n-2}|\int_1^\infty\varepsilon^{n-2}\,dt<\infty$. It is area minimising near infinity because the projection $(t,s,\theta)\mapsto(t,0,\theta)$ is $1$-Lipschitz on $\{|s|<\frac12\}$ and fixes $W$.

For $NT$ and the profile, let $\Omega\supseteq K$ minimise perimeter among compact perturbations containing $K$. Comparing $\Omega$ with $\Omega\cup\{t<b\}$ gives $|\partial\Omega\cap\{t<b\}|\le|S_b\setminus\Omega|$. The vertical projection onto $S_b$ is $1$-Lipschitz, since $f$ is non-increasing in $t$, and it gives the reverse inequality. In the equality case, $\partial\Omega\cap\{t<b\}$ is horizontal and avoids $\{|s|<1\}$, where $f$ is strictly decreasing in $t$. This forces every slice to be empty or full, and hence $\Omega\supseteq\{t<b\}$. So $\Omega=M$: there are no tentacles, and $A(K)$ is not attained. The same projection gives $|\partial\Omega|\ge|S_b|$ for bounded $\Omega\supseteq K$, and hence $A(K)=\lim_{b\to\infty}|S_b|$.

\section{Existence of an exhaustion with small integral of mean curvature}\label{section:smallcurv}

In this section we construct exhaustions whose boundaries are mean convex, or mean concave, up to an error which is small in $L^1$. We first show that one-sided bounds on the mean curvature of separating hypersurfaces give exhaustions $\{U_i\}$ with $\int_{\partial U_i}H^-\to0$, or with $\int_{\partial U_i}H^+\to0$. We then show that these curvature bounds hold in general settings, with constants controlled by the growth of the area profile and by the rate at which $M$ collapses. Unlike the results of Section 2, Theorems \ref{thm:Hminus}--\ref{thm:Habs} below do not need condition $NT$.

 As in Section \ref{sectiontrichotomy}, for compact domains $K\subset\mathring V$ with smooth boundaries we consider the band $X=V\setminus\mathring K$ and the class $\mathcal{C}(X,\partial V,\partial K)$. We write $a_K(V)$ for the least area of a hypersurface in $\mathcal{C}(X,\partial V,\partial K)$.

For a smooth positive function $w$ on $M$ and a hypersurface $\Sigma$, we write $|\Sigma|_w=\int_\Sigma w\,d\mathcal H^{n-1}$, and we define $A_w(K)$ and $a^w_K(V)$ as above, with $|\cdot|_w$ in place of area. Since $|\cdot|_w$ is the area functional of the conformal metric $w^{2/(n-1)}g$, minimisers of $|\cdot|_w$ in $\mathcal{C}(X,\partial V,\partial K)$ exist and are $C^{1,1}$ hypersurfaces. For such a minimiser $\xi$, we call $\xi\cap\mathring X$ the free part of $\xi$, and $\xi\cap\partial K$ and $\xi\cap\partial V$ its inner and outer contact sets. We let $\nu$ be the unit normal of $\xi$ pointing away from $K$. From the first variation, we know that almost everywhere on $\xi$, 
\begin{align}
    H=-\partial_\nu\log w\ \text{ on the free part},\qquad
H\ge-\partial_\nu\log w\ \text{ on }\xi\cap\partial K,\qquad
H\le-\partial_\nu\log w\ \text{ on }\xi\cap\partial V .
\end{align}
Moreover, $H=H_{\partial K}$ almost everywhere on $\xi\cap\partial K$, and $H=H_{\partial V}$ almost everywhere on $\xi\cap\partial V$.

In particular, if $w$ is constant near $\partial K$ then $H\ge0$ on the inner contact set, and if $w$ is constant near $\partial V$ then $H\le0$ on the outer contact set, as in the previous section.

\medskip\noindent\emph{Smoothing.} The minimisers above are only $C^{1,1}$, and we make them smooth by running the mean curvature flow for a short time, as in Section 2. Along the flow, $\partial_tH=\Delta H+(|A|^2+\text{Ric}(\nu,\nu))H$ and $\partial_t\,d\mu=-H^2\,d\mu$. Hence
\begin{align}
\frac{d}{dt}\int_{\xi_t}H^\pm\ \le\ C\int_{\xi_t}H^\pm ,
\end{align}
and an upper bound $H\le b$ with $b\ge0$, or a lower bound $H\ge-b$, changes at most by the factor $e^{Ct}$. Here $C$ depends on a bound for $|A|$ along the flow and on the ambient curvature near $\xi$. These estimates pass to the flow starting from a $C^{1,1}$ hypersurface by approximating it with smooth ones. Since $\xi_t\to\xi$ in $C^1$, in each construction below we may replace $\xi$ by a smooth hypersurface arbitrarily $C^1$-close to it, at an arbitrarily small cost in $\int H^\pm$ and in one-sided bounds on $H$.

\subsection{Exhaustions with small $\int H^-$ or $\int H^+$}\label{subsec:smallcurv}

\begin{defn}
Let $\Lambda\ge0$. We say that $(M,g)$ satisfies condition ($H^-_\Lambda$) if every compact subset of $M$ is contained in a compact domain $V$ with smooth boundary such that $H_{\partial V}\ge-\Lambda$. Similarly, $(M,g)$ satisfies condition ($H^+_\Lambda$) if every compact subset of $M$ is contained in a compact domain $K$ with smooth boundary such that $H_{\partial K}\le\Lambda$.
\end{defn}

Note that here we assume the existence of isolated hypersurfaces far enough into the end satisfying the curvature bound for some $\Lambda$. Hence, when the condition defined above is not met, the manifold must have the infimum of mean curvature going to $-\infty$ (resp. supremum going to $+\infty$) along any sequence of hypersurfaces going into the end and bounding an exhaustion of the manifold.

Our next result shows that the existence of some hypersurfaces far off with any uniform bound on mean curvature is enough to control the $L^1$ error of an exhaustion as desired. Note that here we upgrade an exhaustion with an arbitrarily large but uniform negative lower bound to an exhaustion with almost weakly mean convex boundaries.

\begin{theorem}\label{thm:Hminus}
Let $(M,g)$ be a complete Riemannian manifold of dimension $3\le n\le7$, which satisfies condition ($H^-_\Lambda$) for some $\Lambda>0$. Then there exists a compact exhaustion $\{U_i\}_{i=1}^\infty$ of $M$ with smooth boundaries such that $\int_{\partial U_i}H^-\to0$ as $i\to\infty$.
\end{theorem}

\begin{proof}
Let $K\subset M$ be a compact domain  with smooth boundary, and let $p$ be a point in $K$. We show that for every such compact domain, and every $\varepsilon\in(0,1]$, there is a compact domain $U$ with smooth boundary such that $\int_{\partial U}H^-<\varepsilon$, and which contains any prescribed compact subset of $\mathring K$.

By assumption, there exists some $\Lambda>0$ such that $K$ is contained in the interior of a compact domain $V$ such that $H_{\partial V}\geq -\Lambda$. Let us fix the constants $$k=\varepsilon/(4(1+|\partial K|)),$$ and $$\delta=k\varepsilon/(2\Lambda).$$ Now, we choose some $R>0$ with $K\subset \mathring{B_R}$ such that $a_K(V)<A(K)+\delta$ for every compact domain $V\supset B_R$. By condition $H^-_\Lambda$, we choose a compact domain $V\supset B_{R+5}$ with smooth boundary and $H_{\partial V}\ge-\Lambda$. Let $X=V\setminus\mathring K$.

We now penalise contact with $\partial V$. Let $\tilde\rho$ be a smooth function on $M$ with $|\tilde\rho-d(\cdot,\partial V)|<1$ and $|\nabla\tilde\rho|\le2$. Let $\chi:\mathbb R\to[0,1]$ be smooth and non-increasing, such that
\begin{itemize}
    \item $\chi=1$ on $(-\infty,1]$;
    \item $\chi=0$ on $[3,\infty)$;
    \item $|\chi'|\le1$.
\end{itemize}

Define a function $w$ as
\begin{align}
w=1+k\chi(\tilde\rho),
\end{align}
with the constant $k$ defined earlier. Then $w\ge1$ everywhere, and on $\partial V$ we have $\chi = 1$ and hence $w=1+k$. Also, $w=1$ for any point at distance at least $4$ from $\partial V$ (in particular near $K$ and on $B_R$). Additionally, $|\nabla\log w|\le2k$.

Let $\xi$ be a minimiser of $|\cdot|_w$ in $\mathcal{C}(X,\partial V,\partial K)$. To bound the outer contact set, we compare $\xi$ with a hypersurface which avoids the region where $w>1$. Choose a regular value $s\in(3,4)$ of $\tilde\rho$, and let $V'=\{x\in V:\tilde\rho(x)\ge s\}$. Then $V'$ is a compact domain with smooth boundary, $K\subset B_R\subset V'$, and $w=1$ on $V'$. Let $\xi^*$ be an (unweighted) area minimiser in $\mathcal{C}(V'\setminus\mathring K,\partial V',\partial K)$. Then $\xi^*\in\mathcal{C}(X,\partial V,\partial K)$, and since $w=1$ on $V'$, $|\xi^*|_w=|\xi^*|$. By the choice of $R$,
\begin{align}\label{eqxistarless}
    |\xi^*|_w=|\xi^*|=a_K(V')<A(K)+\delta\le a_K(V)+\delta .
\end{align}
On the other hand, since $w\ge1$ everywhere and $w=1+k$ on $\partial V$,
\begin{align}\label{eqximore}
|\xi|_w\ \ge\ |\xi|+k\,|\xi\cap\partial V|\ \ge\ a_K(V)+k\,|\xi\cap\partial V| .
\end{align}
Since $\xi$ minimises $|\cdot|_w$, and $\xi^*$ is a competitor for $\xi$ in $\mathcal{C}(X,\partial V,\partial K)$, we have $|\xi|_w\le|\xi^*|_w$. Therefore, from Equation \ref{eqxistarless} and Equation \ref{eqximore}, 
\begin{align}
|\xi\cap\partial V|<\delta/k=\varepsilon/(2\Lambda).
\end{align}
We also know $|\xi|\le|\xi|_w\le|\xi^*|\le|\partial K|$, since $\partial K$ is a competitor for $\xi^*$.

Now, we compute the mean curvature bounds for $\xi$. On the free part, $|H_{\xi}|=|\partial_\nu\log w|\le2k$. $w$ is a constant function in a neighbourhood of $\partial K$, so on the inner contact set $\partial K \cap \xi$, we have $H\ge0$. On the outer contact set, $H=H_{\partial V}\ge-\Lambda$. Hence,
\begin{align}
\int_\xi H^-\ \le\ 2k\,|\xi|+\Lambda\,|\xi\cap\partial V|\ <\ \frac{\varepsilon\,|\partial K|}{2(1+|\partial K|)}+\frac\varepsilon2\ <\ \varepsilon .
\end{align}
Running the mean curvature flow for a short time, we obtain a smooth hypersurface, $C^1$-close to $\xi$, with $\int H^-<\varepsilon$. It bounds a compact domain $U$ which contains any prescribed compact subset of $\mathring K$.

Now, we proceed iteratively by choosing a sequence of compact domains $\{K_i\}$ such that the minimiser $\xi_{i-1}$ from the previous iteration lies in the interior of $K_i$, and the $K_i$ exhaust $M$. On the $i^{\text{th}}$ iteration, we ensure $\int_{\xi_i} H^-<\varepsilon/i$ by choosing the constants as $k=\varepsilon/(4i(1+|\partial K_i|))$ and $\delta = k\varepsilon/(2i\Lambda)$. Then the $\xi_i$ bound compact domains $U_i$ such that $\{U_i\}$ is an exhaustion of $M$, and $\int_{\partial U_i} H^-\to 0$.
\end{proof}

Note that the integral $\int H^-$ above is computed in the original metric $g$; the weight $w$ only enters the variational problem.

A similar result holds for the existence of almost weakly mean concave boundary exhaustion.

\begin{theorem}\label{thm:Hplus}
Let $(M,g)$ be a complete Riemannian manifold of dimension $3\le n\le7$ with $A_\infty<\infty$, which satisfies condition ($H^+_\Lambda$) for some $\Lambda>0$. Then there exists a compact exhaustion $\{U_i\}_{i=1}^\infty$ of $M$ with smooth boundaries such that $\int_{\partial U_i}H^+\to0$ as $i\to\infty$.
\end{theorem}

\begin{proof}
The proof is similar to that of Theorem \ref{thm:Hminus} with the roles of $\partial K$ and $\partial V$ exchanged. The one new point is that the inner boundary must also be decoupled, and this is where $A_\infty<\infty$ is used.

We can assume $\Lambda\geq 1$, since $H_\Lambda^+$ implies $H_{\Lambda'}^+$ for any $\Lambda'\geq \Lambda$. Let $\varepsilon\in(0,1]$, set $$k=\varepsilon/(4(A_\infty+1)),$$ and set $$\delta=k\varepsilon/(2\Lambda)$$ as before. Fix a point $p\in M$. Since $A(B_r)\to A_\infty<\infty$, there is $R$ with $A_\infty-A(B_R)<\delta/2$. By condition ($H^+_\Lambda$), choose a compact domain $K\supset B_R$ with smooth boundary and $H_{\partial K}\le\Lambda$. Let $\tilde\rho$ be a smoothing of $d(\cdot,K)$, satisfying $|\tilde\rho-d(\cdot,K)|<1$ and $|\nabla\tilde\rho|\le2$. We again fix a function $\chi:\mathbb R\to[0,1]$ which is smooth and non-increasing, such that
\begin{itemize}
    \item $\chi=1$ on $(-\infty,1]$;
    \item $\chi=0$ on $[3,\infty)$;
    \item $|\chi'|\le1$.
\end{itemize} Taking $w=1+k\chi(\tilde\rho)$, choose a regular value $s\in(3,4)$ of $\tilde\rho$, and let $K'=K\cup\{\tilde\rho\le s\}$. Then $K'$ is a compact domain with smooth boundary, $w=1+k$ near $\partial K$, $w=1$ outside $\mathring K'$, and $|\nabla\log w|\le2k$. Since $B_R\subset K\subset K'$,
\begin{align}\label{eqT1one}
A(B_R)\le A(K)\le A(K')\le A_\infty,
\end{align}
and so,
\begin{align}\label{eqT1two}
A(K')-A(K)<\delta/2 .
\end{align}
Now we choose a compact domain $V$ with $K'\subset\mathring V$ and $a_{K'}(V)<A(K')+\delta/2$. Let $\xi$ be a minimiser of $|\cdot|_w$ in $\mathcal{C}(V\setminus\mathring K,\partial V,\partial K)$, and let $\xi^*$ be an (unweighted) area minimiser in $\mathcal{C}(V\setminus\mathring K',\partial V,\partial K')$. Then $\xi^*$ is a competitor for $\xi$, and it lies where $w=1$. Since $w\ge1$, $w=1+k$ on $\partial K$, and $|\xi|\ge a_K(V)\ge A(K)$, we get
\begin{align}
A(K)+k\,|\xi\cap\partial K|\ \le\ |\xi|_w\ \le\ |\xi^*|_w=a_{K'}(V)\ <\ A(K')+\delta/2\ <\ A(K)+\delta.
\end{align}
This implies, 
\begin{align}
|\xi\cap\partial K|<\delta/k=\varepsilon/(2\Lambda).
\end{align}
We also have $|\xi|\le|\xi^*|<A_\infty+1$, since $\delta<1$. This bound does not depend on $K$, which is why $k$ can be chosen before $K$.

By the first variation, we have that $H=H_{\partial K}\le\Lambda$ on the inner contact set. Since $w=1$ near $\partial V$, $H\le0$ on the outer contact set, and $|H|\le2k$ on the free part. Hence,
\begin{align}
\int_\xi H^+\ <\ \Lambda\,\frac{\delta}{k}+2k\,(A_\infty+1)\ \leq \frac{\varepsilon}{2} + 2k\,(A_\infty+1)  =\ \varepsilon .
\end{align}
Smoothing by a short mean curvature flow and iterating as in Theorem \ref{thm:Hminus} completes the proof.
\end{proof}

Combining the two conditions, we can get an exhaustion with a similar bound on the total curvature.

\begin{theorem}\label{thm:Habs}
Let $(M,g)$ be a complete Riemannian manifold of dimension $3\le n\le7$ with $A_\infty<\infty$, which satisfies conditions ($H^+_\Lambda$) and ($H^-_\Lambda$) for some $\Lambda>0$. Then there exists a compact exhaustion $\{U_i\}_{i=1}^\infty$ of $M$ with smooth boundaries such that $\int_{\partial U_i}|H|\to0$ as $i\to\infty$.
\end{theorem}

\begin{proof}
We assume $\Lambda\geq 1$ without loss of generality. Fix some point $p$ in $M$. Choose $k$, $\delta$, $R$, and $K$ as in the proof of Theorem \ref{thm:Hplus}, and let $\tilde\rho_K$ be a smoothening of the distance functions to $K$  satisfying $|\nabla\tilde\rho_K|\le2$ and $|\tilde\rho_K-d(\cdot,K)|<1$. As before, let $s\in(3,4)$ be a regular value of $\tilde\rho_K$, and set $K'=K\cup\{\tilde\rho_K\leq s\}$. Choose $R'$ such that $K'\subset \mathring{B_{R'}}$ and $a_{K'}(V)<A(K')+\delta/2$ for every compact domain $V\supset B_{R'}$. By condition ($H^-_\Lambda$), choose $V\supset B_{R'+5}$ with $H_{\partial V}\ge-\Lambda$. Let $\tilde\rho_V$ be a smoothening of the distance function to $\partial V$, also satisfying $|\nabla\tilde\rho_V|\le2$ and $|\tilde\rho_V-d(\cdot,\partial V)|<1$. For a  regular value $s'\in(3,4)$ of $\tilde \rho_V$, let $V'=V\cap\{\tilde\rho_V\geq s'\}$. Again, Equation \ref{eqT1one} and Equation \ref{eqT1two} hold exactly as before. We modify the earlier arguments to penalise both boundaries.

Fix the function \begin{align}
w=1+k\chi(\tilde\rho_K)+k\chi(\tilde\rho_V),
\end{align}
where $\chi:\mathbb R\to[0,1]$ is as in the previous theorems. In particular, $\chi=1$ on $(-\infty,1]$, $\chi=0$ on $[3,\infty)$, and $|\chi'|\le1$. Note that the two collars where $k\chi(\tilde\rho_K)>0$ and $k\chi(\tilde\rho_V)>0$ are disjoint by construction. Hence, $w\geq 1$ everywhere, and $w=1+k$ on some neighbourhoods of $\partial K $ and $\partial V$. On $V'\setminus \mathring{K'}$, $w=1$. Also, $|\nabla \log w|\leq 2k$.

We compare the minimiser $\xi$ of $|\cdot|_w$ in $\mathcal{C}(V\setminus\mathring K,\partial V,\partial K)$ with an unweighted area minimiser $\xi^*$ in $\mathcal{C}(V'\setminus\mathring K',\partial V',\partial K')$, exactly as before. Since $w=1$ on $\xi^*$, and $V'\supset B_{R'}$,
\begin{align}
    |\xi^*|_w=|\xi^*|=a_{K'}(V')<A(K')+\delta/2<A(K)+\delta.
\end{align}
Also, 
\begin{align}
    |\xi|\geq a_{K}(V)\geq A(K).
\end{align}
Therefore,
\begin{align}
    |\xi|_w & \geq  |\xi|+k(|\xi\cap \partial K|+|\xi\cap \partial V|)\\ & \geq A(K)+k(|\xi\cap \partial K|+|\xi\cap \partial V|).
\end{align}
Since $|\xi|_w\leq |\xi^*|_w$,
\begin{align}
|\xi\cap\partial K|+|\xi\cap\partial V|<\delta/k,
\end{align}
and
\begin{align}
    |\xi|<A_\infty+1 .
\end{align}
Now, $w$ is constant near both $\partial K$ and $\partial V$, and so $0\le H=H_{\partial K}\le\Lambda$ on the inner contact set, $-\Lambda\le H=H_{\partial V}\le0$ on the outer contact set, and $|H|\le2k$ on the free part. Therefore, 
\begin{align}
    \int_\xi|H|&\leq 2k|\xi|+\Lambda(|\xi\cap\partial K|+|\xi\cap\partial V|)\\
    &<\Lambda\delta/k+2k(A_\infty+1)\\&=\varepsilon.
\end{align}
Smoothing by running the mean curvature flow for a short time, we get the desired hypersurface with total mean curvature less than $\varepsilon$. We iterate as before, choosing a sequence $\{K_i\}$ of domains which exhaust $M$, and arguing similarly to get hypersurfaces $\{\xi_i\}$ such that $\int_{\xi_i}|H|<\varepsilon/i\to 0$. Since these hypersurfaces bound compact domains, we get an exhaustion of $M$ as claimed.
\end{proof}

\subsection{The curvature conditions}\label{sectioncurvatureconditions}

We now define an asymptotic invariant of a Riemannian manifold, and show that conditions ($H^\pm_\Lambda$) hold automatically, under control of this invariant along with the rate of collapse. 

\begin{defn}
For a Riemannian manifold $(M,g)$ the \textit{exponential growth rate of the area profile} is defined as
\begin{align}
h=\liminf_{r\to\infty}\frac1r\log\bigl(1+A(B_r)\bigr)\in[0,\infty].
\end{align}
\end{defn}
Note that $h=0$ when $A_\infty<\infty$.

Fix $\varepsilon_0\in(0,1]$. Take any point $p$ in $M$, and let $d$ be the distance function in $M$ from $p$. Consider a smoothening of $d$, denoted by $\rho$, with $|\rho-d|\le1$ and $|\nabla\rho|\le1+\varepsilon_0$. For $\sigma>0$ we fix the weight functions
\begin{align}
w_\pm=e^{\pm(n-1)\sigma\rho}.
\end{align}
The weights satisfy the gradient inequality
\begin{align}
|\nabla\log w_\pm|\le(1+\varepsilon_0)(n-1)\sigma.
\end{align}

The following results show the existence of exhaustions with pointwise upper or lower bounds on mean curvature of their boundaries. The proof involves choosing weighted bands to find minimisers which satisfy the required curvature constraints. 

\begin{proposition}\label{prop:decaying}
Let $(M,g)$ be a complete Riemannian manifold of dimension $3\le n\le7$. Then for every $\Lambda>h$, every compact subset of $M$ is contained in a compact domain $U$ with smooth boundary such that $H_{\partial U}\le\Lambda$. In particular, if $A_\infty<\infty$ then $M$ satisfies ($H^+_\Lambda$) for every $\Lambda>0$.
\end{proposition}

\begin{proof}
Fix some point $p$ in $M$. Since $\Lambda>h$, there exist constants $\varepsilon_0$ and $\sigma$ such that $h<(n-1)\sigma$ and $(1+\varepsilon_0)(n-1)\sigma<\Lambda$. For these choices of $\varepsilon_0$ and $\sigma$, let $\rho$ be a smoothening of the distance function as described above. Let 
\begin{align}
w=w_-=e^{-(n-1)\sigma\rho}.
\end{align}
We first show that $A_w(K)=0$ for every compact $K$. If $K\subset B_R$ for some radius $R$, and $\Omega\supset B_R$ is a compact domain with $|\partial\Omega|\le A(B_R)+1$, then $\rho\ge R-1$ on $\partial\Omega$. So,
\begin{align}
A_w(K)\ \le\ |\partial\Omega|_w\ \le\ e^{-(n-1)\sigma(R-1)}\bigl(1+A(B_R)\bigr).
\end{align}
Since $(n-1)\sigma>h$, the right side tends to $0$ along a sequence $R\to\infty$.

Now fix a compact domain $K$ with smooth boundary. Since $A_w(K)=0$, $a^w_K(V_i)\to 0$ for every large compact domains $V_i$ exhausting $M$. By the same argument as Proposition \ref{propAinfty0} (or as Theorem $1.2$ in \cite{BZ}) applied to conformal metric $w^{2/(n-1)}g$ in compact bands, we get a smooth hypersurface $\xi$ in some $\mathcal{C}(V_i\setminus \mathring K,\partial V_i,\partial K)$ that does not touch $\partial K$ or $\partial V_i$ for large enough $i$. On this hypersurface,
\begin{align}
H\ \le\ -\partial_\nu\log w\ =\ (n-1)\sigma\,\partial_\nu\rho\ \le\ (1+\varepsilon_0)(n-1)\sigma\ <\ \Lambda .
\end{align}
Hence we get a smooth hypersurface with $H<\Lambda$, bounding a compact domain which contains any prescribed compact subset of $\mathring K$.
\end{proof}

\begin{lemma}\label{lem:tentacles}
Let $\sigma>\kappa$ and $w=w_+$. Let $\Omega$ be a domain whose boundary minimises $|\cdot|_w$ among compact perturbations supported outside a compact set $C$, with $|\partial\Omega\setminus C|_w<\infty$. Then $\partial\Omega$ is bounded.
\end{lemma}

\begin{proof}
Suppose $\partial\Omega$ is unbounded. Fix a point $p$ in $M$. Fix $\varepsilon'>0$ such that $\sigma>\kappa+2\varepsilon'$, so that when $d(q)$ is large, we have $\tilde r(q)\ge e^{-(\kappa+\varepsilon')d(q)}$. Choose points $q_j\in\partial\Omega$ with $d(q_j)\to\infty$, pairwise at distance at least $2$, with $B_1(q_j)\cap C=\emptyset$. Let $r_j=\tilde r(q_j)/2$. On $B_{r_j}(q_j)$ the function $\rho$ varies by at most $2$, so $w$ varies by at most the factor $e^{2(n-1)\sigma}$. Hence, in the chart around $q_j$, $\partial\Omega$ is a quasi-minimiser of Euclidean area, with a constant depending only on $n$ and $\sigma$. The density estimate for quasi-minimisers gives $|\partial\Omega\cap B_{r_j}(q_j)|\ge c\,r_j^{n-1}$ for some constant $c$, and therefore for large enough $j$,
\begin{align}
|\partial\Omega\cap B_{r_j}(q_j)|_w\ \ge\ c\,e^{(n-1)\sigma(d(q_j)-2)}\,r_j^{n-1}\ \ge\ c'\,e^{(n-1)(\sigma-\kappa-\varepsilon')\,d(q_j)}\ \longrightarrow\ \infty .
\end{align}
Here $c'=ce^{-2(n-1)\sigma}2^{-(n-1)}$ is constant. The balls $B_{r_j}(q_j)$ are disjoint and avoid $C$, and so we contradict the assumption that $|\partial\Omega\setminus C|_w<\infty$.
\end{proof}

\begin{proposition}\label{prop:growing}
Let $(M,g)$ be a complete Riemannian manifold of dimension $3\le n\le7$ with a single end and $A_\infty>0$. Then $M$ satisfies ($H^-_\Lambda$) for every $\Lambda>(n-1)\kappa$.
\end{proposition}

\begin{proof}
Fix any point $p$ in $M$. Choose $\varepsilon_0$ and $\sigma$ satisfying $\sigma>\kappa$ and $(1+\varepsilon_0)(n-1)\sigma<\Lambda$, and let $w=w_+$.

First, note that $A_w(B_r)\to\infty$ as $r\to\infty$. Indeed, if $\Omega\supset B_r$ then $\rho\ge r-1$ on $\partial\Omega$, so $|\partial\Omega|_w\ge e^{(n-1)\sigma(r-1)}A(B_r)$, and $A(B_r)$ is positive for large $r$ because $A_\infty>0$.

Next, fix a compact domain $K$ with smooth boundary. We show that $K$ has a weighted hull, that is, a compact domain $\Omega\supset K$ with $|\partial\Omega|_w=A_w(K)$. Take compact domains $V_i$ exhausting $M$, and consider the minimisers $\xi_i$ of $|\cdot|_w$ in $\mathcal{C}(V_i\setminus\mathring K,\partial V_i,\partial K)$. Let the $\xi_i$ bound domains $\Omega_i\supset K$. Suppose first that the $\xi_i$ eventually leave every compact set. Then for large $i$ they enclose $B_r$, so $|\xi_i|_w\ge A_w(B_r)$, while $|\xi_i|_w\le|\partial K|_w$; this contradicts $A_w(B_r)\to\infty$. Hence a subsequence meets a fixed compact set $Z$. Then the limit $\Omega_\infty$ of the enclosed domains has non-empty boundary $\xi_\infty$ near $Z$, and $\Omega_\infty$ minimises $|\cdot|_w$ among compact perturbations containing $K$. By Lemma \ref{lem:tentacles}, $\xi_\infty$ is bounded. Since $M$ has one end, $\Omega_\infty$ is either bounded or contains a neighbourhood of the end. In the second case, it bounds a compact set $W$ disjoint from $\mathring K$, and so we can fill in $W$ to get $\Omega_\infty=M$ by minimality of the limit, contradicting $\xi_\infty\neq\emptyset$. Hence $\Omega_\infty$ is bounded, so $\xi_\infty$ separates $\partial K$ from the end and lies inside some band, where it is a minimiser of $|\cdot|_w$ that does not touch the outer wall. In that case, it is a weighted hull of $K$.

The mean curvature of the boundary therefore satisfies
\begin{align}
H\ \ge\ -\partial_\nu\log w\ \ge\ -(1+\varepsilon_0)(n-1)\sigma\ >\ -\Lambda
\end{align}
almost everywhere, and a short mean curvature flow makes it smooth while keeping $H>-\Lambda$.
\end{proof}

\begin{remark}\label{rem:ends}
Proposition \ref{prop:decaying} uses neither condition $NT$ nor a single end. Proposition \ref{prop:growing} uses a single end only to exclude the case where the limit does not separate $K$ from the end. With several ends, the same argument works, provided every end has $A_\infty>0$.
\end{remark}

\begin{lemma}\label{lem:noncollapsed}
If $M$ is non-collapsed, then $A_\infty>0$.
\end{lemma}

\begin{proof}
Since $M$ is non-collapsed, there exists some $r_0$ such that $\tilde r(q)\geq r_0$. Let $s=r_0/16$, and let $\gamma$ be a ray from a point $p$ in $M$. Given a radius $R$ and a compact domain $\Omega\supset B_R$, the function $f(t)=|\Omega\cap B_s(\gamma(t))|/|B_s(\gamma(t))|$ is continuous, equal to $1$ for $t\le R-s$, and equal to $0$ for $t$ large. So $f(t^*)=\frac12$ for some $t^*\ge R-s$. Let $x=\gamma(t^*)$. In the chart on $B_{r_0}(x)$, the metric ball $B_s(x)$ lies in the Euclidean ball of radius $2s$ about $x$, which lies in the image of $B_{4s}(x)$. The relative isoperimetric inequality in this Euclidean ball gives $|\partial\Omega\cap B_{4s}(x)|\ge c(n)\,s^{n-1}$. Hence $A(B_R)\ge c(n)\,s^{n-1}$ for all large $R$.
\end{proof}

\begin{proposition}\label{prop:noncollapsed}
Let $(M,g)$ be a complete, non-collapsed Riemannian manifold of dimension $3\le n\le7$, with any number of ends. Then for every $\Lambda>0$, every compact set is contained in a compact domain with smooth boundary on which $H>-\Lambda$. In particular, $M$ satisfies ($H^-_\Lambda$) for every $\Lambda>0$.
\end{proposition}

\begin{proof}
Let $\tilde r(q)\ge r_0>0$ on all $q$ in $M$. Fix $\lambda\in(0,\Lambda)$ and a compact domain $K\subset M$ with smooth boundary. We define the functional $G(\Omega)=|\partial\Omega|+\lambda|\Omega|$. Take compact domains $V_i$ exhausting $M$, and minimisers $\Omega_i$ of $G$ over $K\subseteq\Omega\subseteq V_i$. As in the case of area minimisers, the $\Omega_i$ have $C^{1,1}$ boundaries, with $H=-\lambda$ on the free part, $H\ge-\lambda$ on the inner contact set, and $H\le-\lambda$ on the outer contact set. Since $K$ is admissible, $|\partial\Omega_i|\le G(K)$ and $|\Omega_i|\le G(K)/\lambda$. Hence a subsequence of $\{\Omega_i\}$ converges to a domain $\Omega_\infty\supseteq K$ with
\begin{align}
|\partial\Omega_\infty|\le G(K),\qquad |\Omega_\infty|\le G(K)/\lambda ,
\end{align}
which minimises $G$ among compact perturbations containing $K$.

Since $M$ is non-collapsed and $\Omega_\infty$ minimises $G$ amongst compact perturbations containing $K$, for any point $x\in \partial\Omega_\infty$, we can compare $\Omega_\infty\cup B(x,r)$ and $\Omega_\infty\setminus B(x,r)$ with $\Omega_\infty$ for a small radius $r$. Thus there are constants $s_0,c_0>0$, depending only on $n$, $r_0$ and $\lambda$, such that $| B(x,r)\setminus \Omega_\infty|\geq c_0r^n$ and $|\Omega_\infty\cap B(x,r)|\geq c_0r^n$ for every $x\in\partial\Omega_\infty$ with $d(x,K)\ge s_0$ and $r\leq s_0$. Then the relative isoperimetric inequality, as in Lemma \ref{lem:noncollapsed}, implies $|\partial\Omega_\infty\cap B_{s_0}(x)|\ge c_1$ for every $x\in\partial\Omega_\infty$ with $d(x,K)\ge s_0$ for some constant $c_1>0$. If $\partial\Omega_\infty$ were unbounded, it would contain infinitely many points at mutual distance at least $2s_0$ and at distance at least $s_0$ from $K$. The density estimate would then give $|\partial\Omega_\infty|=\infty$. If instead $\Omega_\infty$ were unbounded with bounded boundary, it would contain an unbounded component of the complement of a compact domain containing $\partial\Omega_\infty$. Such a component contains infinitely many disjoint balls of radius $r_0/2$, so $|\Omega_\infty|=\infty$. Since $|\partial \Omega_\infty|$ and $|\Omega_\infty|$ are bounded above, both conclusions give rise to contradictions, and so $\Omega_\infty$ is bounded.

Hence, $\Omega_\infty$ is bounded and minimises $G$ among all bounded sets containing $K$. Hence it is a minimiser in any band $\mathcal{C}(V\setminus\mathring K,\partial V,\partial K)$ with $\bar\Omega_\infty\subset\mathring V$, and also it does not touch $\partial V$. By the sign conditions, $H\ge-\lambda$ almost everywhere on $\partial\Omega_\infty$. A short mean curvature flow gives a smooth hypersurface with $H>-\Lambda$, which bounds a compact domain containing any prescribed compact subset of $\mathring K$. Iterating this process with a suitable sequence $\{K_i\}$ of domains exhausting $M$ as in the earlier arguments finishes the proof.
\end{proof}

We summarise the different existence criteria obtained above in Theorem \ref{cor:L1} for convenience.

\subsection*{Proof of Theorem \ref{cor:L1}}
Part (1) follows from Proposition \ref{prop:decaying} and Theorem \ref{thm:Hplus}. Part (2) is implied by Proposition \ref{prop:growing} and Theorem \ref{thm:Hminus}. Part (3) follows from Proposition \ref{prop:decaying} and Proposition \ref{prop:growing}, combined with Theorem \ref{thm:Habs}. For part (4), Proposition \ref{prop:noncollapsed} and Theorem \ref{thm:Hminus} give the first conclusion, while Propositions \ref{prop:noncollapsed} and \ref{prop:decaying} along with Theorem \ref{thm:Habs} give the second conclusion.
\qed

\begin{remark}\label{rem:failure}
For a single-ended $M$, Theorem \ref{cor:L1} shows that $\int_{\partial U_i}H^+\to0$ can fail for every exhaustion only if $A_\infty=\infty$, and $\int_{\partial U_i}H^-\to0$ can fail for every exhaustion only if $A_\infty=0$ or $\kappa=\infty$. In particular, both can fail together only if $A_\infty=\infty$ and $\kappa=\infty$.
\end{remark}

\subsection{Almost $c$-mean convex exhaustions}\label{subsec:cconvex}

For some constant $c>0$, we study the existence of exhaustions whose boundaries are almost $c$-convex. That is, we show that under certain suitable conditions, there exist exhaustions $\{U_i\}_{i=1}^\infty$, such that on the boundaries, the integral $\int_{\partial U_i}(H-c)^-$ can be made arbitrarily small. Note that here we ensure that the $L^1$ error of the mean curvature is arbitrarily small under the constant $c$, but this does not imply that the hypersurfaces have positive mean curvature everywhere.

For $c\in\mathbb R$ and a bounded set $\Omega\subset M$ of finite perimeter, write
\begin{align}
F_c(\Omega)=|\partial\Omega|-c\,|\Omega| ,
\end{align}
where $|\partial\Omega|$ denotes the perimeter of $\Omega$ and $|\Omega|$ its volume. As in the previous sections, for a compact domain $K\subset M$ with smooth boundary, we define
\begin{align}
A^c(K)=\inf\bigl\{F_c(\Omega):\ \Omega\subset M \text{ bounded, of finite perimeter, with } K\subseteq\Omega\bigr\}.
\end{align}
By approximation, the infimum is the same if it is taken over compact domains $\Omega$ with smooth boundary and $K\subset\mathring\Omega$. For compact domains $K\subset\mathring V$ with smooth boundaries, and $X=V\setminus\mathring K$, we define
\begin{align}
a^c_K(V)=\inf\bigl\{F_c(\Omega):\ \Omega \text{ of finite perimeter, with } K\subseteq\Omega\subseteq V\bigr\}.
\end{align}
Equivalently, $a^c_K(V)$ is the infimum of $|\xi|-c\,|\Omega_\xi|$ over hypersurfaces $\xi\in\mathcal{C}(X,\partial V,\partial K)$, where $\Omega_\xi$ denotes the region bounded by $\xi$ that contains $K$. This infimum is attained by a set whose boundary is a $C^{1,1}$ hypersurface. For a smooth positive function $w$, the weighted quantities $A^c_w(K)$ and $a^{c,w}_K(V)$ are defined in the same way, with $F_c$ replaced by $F^w_c(\Omega)=|\partial\Omega|_w-c\,|\Omega|$.

The minimisers of $F^w_c$ satisfy mean curvature constraints identical to the area minimisers in the earlier sections, with the obvious changes. Minimisers of $F^w_c$ in $\mathcal{C}(X,\partial V,\partial K)$ exist and are $C^{1,1}$. Almost everywhere on such a minimiser $\xi$,
\begin{align}
wH+\partial_\nu w=c\ \text{ on the free part},\qquad
wH+\partial_\nu w\ge c\ \text{ on }\xi\cap\partial K,\qquad
wH+\partial_\nu w\le c\ \text{ on }\xi\cap\partial V,
\end{align}
and $H=H_{\partial K}$, $H=H_{\partial V}$ on the respective contact sets. Smoothing changes $\int(H-c)^-$ by an arbitrarily small amount. Finally, if $A^c(K)>-\infty$, then for every $\delta>0$ there is $R$ such that $A^c(K)\le a^c_K(W)<A^c(K)+\delta$ for every compact domain $W\supset B_R$.

\begin{lemma}\label{lem:isoperimetric}
Suppose that $M$ satisfies a linear isoperimetric inequality with constant $c'>c>0$, and let $K$ be a compact domain. Then $A^c(K)\ge-C'$. Moreover, if $w\ge1$ and $w=1$ near $\partial K$, then every $\Omega$ with $K\subseteq\Omega$ and $F^w_c(\Omega)\le F^w_c(K)$ satisfies
\begin{align}
|\partial\Omega|\ \le\ \alpha_K:=\frac{F_c(K)+cC'/c'}{1-c/c'} .
\end{align}
\end{lemma}

\begin{proof}
For every compact domain $\Omega$ we have $F_c(\Omega)= |\partial \Omega| - c|\Omega| \geq (c'-c)|\Omega|-C'\ge-C'$, which shows the first claim. 

For the second claim, $w\ge1$ and $F^w_c(K)=F_c(K)$ give $|\partial\Omega|\le F^w_c(\Omega)+c|\Omega|\le F_c(K)+c|\Omega|$. On the other hand since $M$ satisfies the linear isoperimetric inequality with $c'$, $|\Omega|\le(|\partial\Omega|+C')/c'$. Since $c<c'$,
\begin{align}
    |\partial \Omega|\leq c\frac{(|\partial\Omega|+C')}{c'}+F_c(K).
\end{align}
On rearranging, we get the claimed bound.
\end{proof}

\begin{theorem}\label{thm:cconvex}
Let $(M,g)$ be a complete Riemannian manifold of dimension $3\le n\le7$, and let $c>0$. Suppose that $M$ satisfies a linear isoperimetric inequality with constant $c'>c$, and condition ($H^-_\Lambda$) for some $\Lambda>0$. Then there exists a compact exhaustion $\{U_i\}_{i=1}^\infty$ of $M$ with smooth boundaries such that $\int_{\partial U_i}(H-c)^-\to0$.
\end{theorem}

\begin{proof}
The proof follows the idea of Theorem \ref{thm:Hminus}. It suffices for us to show that for every compact domain $K$ with smooth boundary and every $\varepsilon\in(0,1]$, there is a compact domain $U$ with smooth boundary, containing any prescribed compact subset of $\mathring K$, such that $\int_{\partial U}(H-c)^-<\varepsilon$. With $\alpha_K$ as in Lemma \ref{lem:isoperimetric}, set
\begin{align}
k=\frac{\varepsilon}{2(2+c)(1+\alpha_K)},
\end{align}
and
\begin{align}
\delta=\frac{k\varepsilon}{2(\Lambda+c)} .
\end{align}
Fix a point $p$ in $K$. Since $A^c(K)$ is finite, we can choose $R$ with $K\subset B_R$ and $a^c_K(W)<A^c(K)+\delta$ for every compact domain $W\supset B_R$. By ($H^-_\Lambda$), choose a compact domain $V\supset B_{R+5}$ with smooth boundary and $H_{\partial V}\ge-\Lambda$. Define a function $\tilde \rho$ with $|\tilde\rho-d(\cdot,\partial V)|<1$ and $|\nabla\tilde\rho|\le2$, and define a cut-off function $\chi$ exactly as in Theorem \ref{thm:Hminus}. Again, we define the weight $w$ as $w=1+k\chi(\tilde\rho)$. Then $1\le w\le1+k$, $w=1+k$ near $\partial V$, $w=1$ at distance at least $4$ from $\partial V$ (in particular near $K$), and $|\nabla\log w|\le2k$. Let $\tilde\Omega$ minimise $F^w_c$ over $K\subseteq\Omega\subseteq V$, and let $\xi$ be its boundary. Since $K$ is admissible, $F^w_c(\tilde\Omega)\le F_c(K)$, so $|\xi|\le \alpha_K$ by Lemma \ref{lem:isoperimetric}.

Let $V'=\{x\in V:\tilde\rho(x)\ge s\}$ for a regular value $s\in(3,4)$ of $\tilde\rho$. Then $B_R\subset V'$ and $w=1$ on $V'$. Let $\Omega^*$ minimise $F_c$ over $K\subseteq\Omega\subseteq V'$. Then $\Omega^*$ is admissible in the weighted problem, and its boundary lies where $w=1$, so $F^w_c(\Omega^*)=a^c_K(V')<A^c(K)+\delta$. On the other hand, $w\ge1$, $w=1+k$ on $\partial V$, and the volume terms of $F^w_c$ and $F_c$ coincide. So,
\begin{align}
A^c(K)+k\,|\xi\cap\partial V|\ \le\ F_c(\tilde\Omega)+k\,|\xi\cap\partial V|\ \le\ F^w_c(\tilde\Omega)\ \le\ F^w_c(\Omega^*)\ <\ A^c(K)+\delta ,
\end{align}
and so $|\xi\cap\partial V|<\delta/k=\varepsilon/(2(\Lambda+c))$.

On the free part of $\xi$ we have $H=c/w-\partial_\nu\log w$, and since $|1/w-1|\le k$, this gives $|H-c|\le(2+c)k$. On $\xi\cap\partial K$, $w$ is constant nearby, so $H\ge c$. On $\xi\cap\partial V$, $H=H_{\partial V}\ge-\Lambda$, so $(H-c)^-\le\Lambda+c$. Hence
\begin{align}
\int_\xi(H-c)^-\ \le\ (2+c)k\,\alpha_K+(\Lambda+c)\,|\xi\cap\partial V|\ <\ \frac{\varepsilon \alpha_K}{2(1+\alpha_K)}+\frac\varepsilon2\ <\ \varepsilon .
\end{align}
Smoothing $\xi$ gives a smooth hypersurface with $\int(H-c)^-<\varepsilon$, which bounds a compact domain containing any prescribed compact subset of $\mathring K$.
\end{proof}

\subsection{Proof of Corollary \ref{cor:cconvex}}
If $M$ is non-collapsed, ($H^-_\Lambda$) holds for every $\Lambda>0$ by Proposition \ref{prop:noncollapsed}. If $M$ has a single end and $\kappa<\infty$, by Proposition \ref{prop:growing}, it holds for every $\Lambda>(n-1)\kappa$. The hypothesis $A_\infty>0$ of that proposition is satisfied in this case because of the isoperimetric inequality, since for $\Omega\supset B_r$, we have $|\partial\Omega|\ge c'|B_r|-C'\to\infty$ (since $|B_r|\to \infty$). Now we apply Theorem \ref{thm:cconvex}.
\qed

\begin{remark}
For $c\le0$ the statement follows from the case $c=0$, since $(H-c)^-\le H^-$. The analogous statement for $(H-c)^+$ would require penalising the inner obstacle, and hence a finite limit of $A^c(B_r)$. Under a linear isoperimetric inequality with constant $c'>c$ this limit is infinite if the manifold has infinite volume, since $A^c(B_r)\ge(c'-c)|B_r|-C'$. For $c>h$, however, Proposition \ref{prop:decaying} already gives exhaustions with $H<c$ everywhere. A positive Cheeger constant gives the hypothesis with $C'=0$. In $\mathbb H^n$ every $c<n-1$ is allowed, although there large spheres already have $H>n-1$. We do not know whether exhaustions with $H\ge c$ everywhere exist for $c\in(0,c')$, or whether the constant $c'$ is sharp.
\end{remark}

\end{document}